\documentclass[11pt]{amsart}
\usepackage{extsizes}
\usepackage{amsfonts, mathrsfs,xcolor}
\usepackage{amssymb}
\usepackage{latexsym}
\usepackage{color, xcolor}
\usepackage{hyperref}
\usepackage{soul}
\usepackage [latin1]{inputenc}

\input{xypic}
\xyoption{all}

\usepackage{graphicx,amsmath,amssymb,amsthm,array,enumerate,mhsetup,mathtools,comment,multicol,tabularx, dutchcal}

\newcommand{\authorfootnotes}{\renewcommand\thefootnote{\@fnsymbol\c@footnote}}%
\makeatother

\newtheorem{thm}{Theorem}[section]
\newtheorem{lemma}[thm]{Lemma}
\newtheorem{proposition}[thm]{Proposition}
\newtheorem{corollary}[thm]{Corollary}

\newtheorem{remark}[thm]{Remark}

\numberwithin{equation}{section}

\newcommand{\norm}[1]{\ensuremath{N(\ideal{#1})}}

\begin{document}

\def\q{\mathfrak{q}}
\def\p{\mathfrak{p}}
\def\l{\mathfrak{l}}
\def\u{\mathfrak{u}}
\def\a{\mathfrak{a}}
\def\b{\mathfrak{b}}
\def\m{\mathfrak{m}}
\def\n{\mathfrak{n}}
\def\r{\mathfrak{r}}
\def\c{\mathfrak{c}}
\def\d{\mathfrak{d}}
\def\e{\mathfrak{e}}
\def\k{\mathfrak{k}}
\def\z{\mathfrak{z}}
\def\h{{\mathfrak h}}
\def\gl{\mathfrak{gl}}
\def\sl{\mathfrak{sl}}

\def\Ext{{\rm Ext}}
\def\Hom{{\rm Hom}}
\def\Ind{{\rm Ind}}

\def\res{\mathop{Res}}

\def\GL{{\rm GL}}
\def\SL{{\rm SL}}
\def\SO{{\rm SO}}
\def\O{{\rm O}}

\def\R{\mathbb{R}}
\def\C{\mathbb{C}}
\def\Z{\mathbb{Z}}
\def\N{\mathbb{N}}
\def\Q{\mathbb{Q}}
\def\A{\mathbb{A}}
\def\D{\mathbb{D}}
\def\Re{\text{Re}}
\def\Im{\text{Im}}

\def\w{\wedge}

\def\Cat{\mathcal{C}}
\def\HC{{\rm HC}}
\def\HCat{\Cat^\HC}
\def\proj{{\rm proj}}

\def\to{\rightarrow}
\def\To{\longrightarrow}

\def\1{1\!\!1}
\def\dim{{\rm dim}}

\def\th{^{\rm th}}
\def\isom{\approx}

\def\CE{\mathcal{C}\mathcal{E}}
\def\E{\mathcal{E}}

\def\dis{\displaystyle}
\def\f{{\bf f}}                 
\def\g{{\bf g}}
\def\T{{\rm T}}              
\def\omegatil{\tilde{\omega}}  
\def\H{\mathcal{H}}            
\def\Dif{\mathfrak{D}}      
\def\W{W^{\circ}}           
\def\Whit{\mathcal{W}}      
\def\ringO{\mathcal{O}}     
\def\S{\mathcal{S}}      
\def\M{\mathcal{M}}      
\def\K{{\rm K}}          
\def\h{\mathfrak{h}} 
\def\norm{{\rm N}}       
\def\trace{{\rm Tr}} 
\def\ctilde{\tilde{C}}

\title{Large Sums of Fourier Coefficients of Cusp Forms}

\author{Claire Frechette}
\address[Claire Frechette]{Boston College, Department of Mathematics, Chestnut Hill, MA 02467, USA}
\email{frechecl@bc.edu}

\author{Mathilde Gerbelli-Gauthier}
\address[Mathilde Gerbelli-Gauthier]{McGill University, Department of Mathematics, Montr\'eal, QC H3A 2K6, Canada}
\email{mathilde.gerbelli-gauthier@mcgill.ca}

\author{Alia Hamieh}
\address[Alia Hamieh]{University of Northern British Columbia, Department of Mathematics and Statistics, Prince George, BC V2N 4Z9, Canada}
\email{alia.hamieh@unbc.ca}

\author{Naomi Tanabe}
\address[Naomi Tanabe]{Bowdoin College,  Department of Mathematics, Brunswick, ME 04011, USA}
\email{ntanabe@bowdoin.edu}


\date{\today}
\thanks{The research of Claire Frechette is supported by NSF grant DMS-2203042, and the research of Alia Hamieh is  supported by NSERC Discovery grant RGPIN-2018-06313.}

\begin{abstract}
Let $N$ be a fixed positive integer, and let $f\in S_k(N)$ be a primitive cusp form given by the Fourier expansion
$f(z)=\sum_{n=1}^{\infty} \lambda_f(n)n^{\frac{k-1}{2}}e(nz)$. We consider the partial sum $S(x,f)=\sum_{n\leq x}\lambda_f(x)$. It is conjectured that $S(x,f)=o(x\log x)$ in the range $x\geq k^{\epsilon}$. Lamzouri proved in \cite{Lamzouri} that this is true under the assumption of the Generalized Riemann Hypothesis (GRH) for $L(s,f)$. In this paper, we prove that this conjecture holds under a weaker assumption than GRH.  In particular, we prove that given $\epsilon>(\log k)^{-\frac{1}{8}}$ and $1\leq T\leq (\log k)^{\frac{1}{200}}$, we have $S(x,f)\ll \frac{x\log x}{T}$ in the range $x\geq k^{\epsilon}$  provided that $L(s,f)$ has no more than $\epsilon^2\log k/5000$ zeros in the region $\left\{s\,:\, \Re(s)\geq \frac34, \, |\Im(s)-\phi| \leq \frac14\right\}$ for every real number $\phi$ with $|\phi|\leq T$.  

 \end{abstract}

\maketitle

\section{Introduction}

Let $f\in S_k(N)$ be a primitive cusp form given by the Fourier expansion
\[f(z)=\sum_{n=1}^{\infty} \lambda_f(n)n^{\frac{k-1}{2}}e(nz).\] 
A standard argument shows that $\lambda_f(n)\ll n^{\frac12}$ for all $n\geq1$. One of the deepest theorems in the theory of modular forms is the upper bound \begin{equation*}\label{deligne}
|\lambda_f(n)| \leq \tau(n)
\end{equation*} 
for any $n\geq1$, where $\tau(n)$ is the number of positive divisors of $n$. This bound was established by Deligne \cite{Deligne} in 1974, thus settling Ramanujan's conjecture for classical modular forms.

In order to better understand the asymptotic behaviour of these coefficients, a common approach is to study them on average. A standard bound in this direction was established by Hecke \cite{hecke} in 1927 when he proved that 
\[S(x,f)\vcentcolon=\sum_{n\leq x}\lambda_f(n)\ll_{f} x^{\frac{1}{2}}\] 
for $x$ large enough. Subsequent work of Walfisz \cite{walfisz} combined with Deligne's bound yields the bound \begin{equation}\label{eq:walfisz}S(x,f)\ll_{f} x^{\frac{1}{3}+\epsilon}\end{equation} for any $\epsilon>0$. Hafner and Ivic \cite{Hafner-Ivic} improved upon this result by removing the factor $x^{\epsilon}$ from \eqref{eq:walfisz}. The implicit constants in these estimates
depend on the modular form $f$. In many applications, one seeks asymptotic estimates for $S(x,f)$ that are uniform in the level aspect or the weight aspect of the underlying modular form. If $f$ is a primitive cusp form of weight $k$ and a fixed level $N$, one could use Perron's formula and the convexity bound for $L(s,f)$ to prove that \[S(x,f)\ll (xk)^{\frac12+\epsilon},\] as $x,k\to\infty$. This implies that \begin{equation}\label{eq:convexity}S(x,f)=o(x\log x)\end{equation} in the range $x>k^{1+\epsilon}$. In fact, using subconvexity bounds for $L(s,f)$, one sees that \eqref{eq:convexity} is valid in the wider range $x>k^{1-\delta}$ for some $\delta>0$.
For  a primitive cusp form $f$ in $S_{k}(1)$, Lamzouri \cite[Corollary~1.2]{Lamzouri} proved that \eqref{eq:convexity} holds in the range $\log x/\log\log k\to\infty$ assuming the GRH for $L(s,f)$. He also proved unconditionally that this range in $x$ is best possible \cite[Corollary 1.4]{Lamzouri}. Lamzouri's work is a $\mathrm{GL}(2)$ analogue of the work of Granville and Soundararajan  \cite{GS1}  on large character sums in which they proved that, for  a primitive character $\chi\mod q$, we have 
$\sum_{n\leq x}\chi(n)=o(x),$ as $\log x/\log\log q\to\infty$ assuming the GRH for $L(s,\chi)$.  In \cite{GS3}, they showed that this asymptotic holds under the weaker assumption that ``100\%" of the zeros of $L(s,f)$ up to height $\frac14$
lie on the critical line. To achieve this goal, Granville and Soundararajan established concrete connections between large character sums and zeros of $L(s,\chi)$. 
The work in this paper is  motivated by the aforementioned papers of Granville and Soundararajan. In fact, our main results stated below are $\mathrm{GL}(2)$ analogues of \cite[Theorem~1.3]{GS3} and \cite[Corollary~1.2]{GS3}.
\begin{thm}\label{thm:main}
	Let $f\in S_k(N)$ be a primitive cusp form, and let $\exp(\sqrt{\log k})\leq x \leq \sqrt{k}$ be such that $|S(x, f)|=\frac{x\log x}{Q}$ where $1\leq Q\leq (\log x)^{1/100}$. Then there exists an absolute constant $C>0$ such that for some real number $\phi$ with $|\phi|\leq CQ$ and any parameter $100CQ^3\leq L\leq 40 \log x$, the region
	\begin{equation}\label{region1}\left\{s\,:\, |s-(1+i\phi)|<\frac{L\log(k(\log k)^{1/\gamma})}{(\log x)^2}\right\},\end{equation}
	with $\gamma = L/100\log x$, contains at least $L/625$ zeroes of $L(s, f)$. 
\end{thm}
The proof of Theorem~\ref{thm:main} is given  in Section~\ref{sec:proof}. Furthermore, we can derive the following corollary from this theorem.
\begin{corollary} \label{cor:main} Let $f\in S_k(N)$ be a primitive cusp form. Let $\epsilon$ and $T$ be real numbers with $\epsilon\geq (\log k)^{-1/8}$ and $1\leq T\leq (\log k)^{1/200}$. Suppose that for every real $\phi$ with $|\phi|\leq T$ the region
\begin{equation}\label{region2} 
\left\{s\,:\, \Re(s)\geq \frac34, \, |\Im(s)-\phi| \leq \frac14\right\} 
\end{equation}
contains no more than $\epsilon^2\log k/5000$ zeroes of $L(s, f)$. Then for all $x\geq k^\epsilon$, we have
\[\left| \sum_{n\leq x}\lambda_f(n)\right|\ll \frac{x\log x}{T}.\]
\end{corollary}
\begin{proof}
We choose $L=\epsilon^2\log k/8$ in Theorem \ref{thm:main}. For any $x\geq k^\epsilon$, we observe that the region given in \eqref{region1} is contained in the region given in \eqref{region2} since
\[ \frac{L\log(k(\log k)^{1/\gamma})}{(\log x)^2}\leq \frac{2L\log k}{(\log x)^2}=\frac28\left(\frac{\epsilon \log k}{\log x}\right)^2 \leq \frac14.\]
 \end{proof}

\begin{remark}
Let $N(T,f)$ be the number of zeros $\rho=\beta+it$ of $L(s,f)$ such that $0\leq \beta\leq 1$ and $|t|\leq T$. By \cite[Theorem~5.38]{IwaniecKowalski}, we have \[N(T,f)=\frac{T}{\pi}\log\frac{NT^2}{(2\pi e)^2}+O\left(\log\left(N(T+k)\right)\right),\] for $T\geq 2$.
Hence, $L(s,f)$ has $O(\log k)$ zeros in the critical strip up to height $(\log k)^{\frac{1}{200}}$. Corollary \ref{cor:main} implies that if \eqref{eq:convexity} is false for $x=k^{\epsilon}$, then a positive proportion $\gg \epsilon^2$ of these zeros lie off the critical line.  
\end{remark}

In this paper we adapt the strategy of proof employed in \cite{GS3}. The main idea can be found in Proposition \ref{prop} where we prove that the inequality $S(e^y,f)\geq e^yy^{1-\frac{1}{100}}$ yields a lower bound for a certain sum taken over non-trivial zeros of $L(s,f)$. This is accomplished by first employing Lemma \ref{lem:estimate average} which yields a relation between $S(e^y,f)$ and $S(e^y,f,\phi)=\sum_{n\leq e^y}\lambda_f(n)n^{-i\phi}$ for some real number $\phi$ via various applications of results from the theory of mean values of multiplicative functions such as Corollary \ref{cor:lipschitz} (Lipschitz Theorem) and Proposition \ref{mvfbound} (Hal\'asz's Theorem). Then we use Lemma \ref{lem:L-average} which is an application of Plancherel's formula relating an integral expression involving $L(1-\gamma+i(\phi+\xi),f)$ (for some $0<\gamma\leq\frac12$) as $\xi$ varies in $\mathbb{R}$ with an integral expression involving the twisted partial sums $S(e^{y},f,\phi)$ as $y$ varies in $\mathbb{R}$. To tie these relations together and establish the conclusion of Proposition \ref{prop}, we resort to Lemma \ref{lem:L-upperbound} which uses the classical explicit formula for $L(s,f)$ to furnish an upper bound for $L(1-\gamma+i(\phi+\xi),f)$ in terms of a sum taken over the non-trivial zeros of $L(s,f)$.

The paper is structured as follows: Section~\ref{sec:preliminaries} provides some analytic tools and preliminaries. In Section~\ref{sec:pretentious}, we delve into some key estimates regarding mean values of divisor-bounded multiplicative functions. These estimates are used in Section~\ref{sec:lemmas} to establish a couple of lemmas, which are crucial to proving Proposition~\ref{prop}. Finally, the proof of the main theorem is presented in Section~\ref{sec:proof}. 

\vskip .1in
\noindent {\bf{Conventions and Notation.}} In this work, we adopt the following conventions and notation. Given two functions $f(x)$ and $g(x)$ we write $f(x) \ll g(x)$, $g(x) \gg f(x)$ or $f(x) = O(g(x))$ to mean there exists some positive constant $M$ such that $|f(x)| \leq M |g(x)|$ for $x$ large enough. The notation $f(x) \asymp g(x)$ is used when both estimates $f(x) \ll g(x)$ and $f(x) \gg g(x)$ hold simultaneously. 
We write $f(x) = o(g(x))$ when $g(x) \neq 0$ for
sufficiently large $x$ and $\displaystyle{\lim_{x\to\infty} \frac{f(x)}{g(x)} = 0}$.  The letter $p$ will be exclusively used to represent a prime number. 

\section{Analytic Tools and Preliminaries}\label{sec:preliminaries}
Throughout this paper, the set of all primitive cusp forms in $S_k(N)$ is denoted as $H_k(N)$. The $L$-function associated to $f\in H_k(N)$ is given by the Dirichlet series 
\[L(s,f)=\sum_{n=1}^{\infty}\frac{\lambda_f(n)}{n^s}\] which is absolutely convergent for $\operatorname{Re}(s)>1$. In this region, the $L$-function can be represented as the Euler product
\begin{align}\label{eqn:L-fn}
L(s,f) &= \prod_{p|N}\left(1-\lambda_f(p)p^{-s}\right)^{-1}\prod_{p \nmid N}\left(1-\lambda_f(p)p^{-s}+p^{-2s}\right)^{-1}\nonumber \\
&= \prod_{p}  \left(1-\frac{\alpha_{1, f}(p)}{p^s} \right)^{-1}\left(1-\frac{\alpha_{2, f}(p)}{p^s} \right)^{-1},
\end{align}
where $\alpha_{1, f}(p)$ and $\alpha_{2, f}(p)$ are referred to as the $p$-th local parameters of $f$. Since the Ramanujan conjecture for classical modular forms is known, thanks to Deligne's work, we have $|\alpha_{j, f}(p)|=1$ for all $p\nmid N$. 

The completed $L$-function which we denote by $\Lambda(s, f)$ (see below) can be analytically continued to an entire function of order 1 and  satisfies a functional equation that relates its value at $s$ to its value at $1-s$.   In particular, we have  the Hadamard factorization \begin{align} \label{Hadamardprod} 
  \Lambda(s, f) \vcentcolon &=N^{s/2} \cdot 2^{(3-k)/2}\sqrt{\pi} (2\pi)^{-s}  \Gamma \left (s + \frac{k-1}{2} \right) L(s,f) \nonumber \\
	&=  e^{A+Bs}\prod_{\rho} \left(1-\frac{s}{\rho}\right) e^{s/\rho},  
 \end{align}
for some $A,B\in \C$, where $B$ satisfies the property $\Re(B) = \sum_\rho - \Re\left(\frac{1}{\rho}\right)$. See  \cite[Equations~5.4,\;5.23,\; and\;5.86]{IwaniecKowalski} for the details. 

To conclude this section, we prove a useful lemma that will be instrumental in the upcoming sections.
\begin{lemma}\label{lem:L-upperbound}
Suppose $\gamma$ is a real number such that $0< \gamma \leq \frac{1}{2}$ and $t$ is any real number. Then
\begin{equation*}\label{eqn:L-upperbound}
\left|L ( 1-\gamma + it,f) \right| \ll \frac{1}{\gamma^2} \emph{exp}\left(\sum_\rho \frac{2\gamma^2}{|1+\gamma + it - \rho|^2} \right).
\end{equation*}
\end{lemma}
\begin{proof}
Let $s_0 = 1 + \gamma + it$ and $s_1 = 1- \gamma + it$. We begin by considering the ratio of $L$-functions evaluated at $s_0$ and $s_1$ using the functional equation:
\begin{align*}
\left|\frac{L(s_1,f)}{L(s_0,f)} \right| 
&= \left|\frac{\Lambda(s_1,f)}{\Lambda(s_0,f)}\right| \left(\frac{N}{4\pi^2}\right)^\gamma \left| \frac{\Gamma\left(s_0 + \frac{k-1}{2} \right)}{\Gamma\left(s_1 + \frac{k-1}{2} \right)} \right|.
\end{align*}
An application of Stirling's formula, together with the Hadamard factorization \eqref{Hadamardprod}, yields 
\begin{align}\label{asympL}
\left|\frac{L(s_1,f)}{L(s_0,f)} \right| \asymp (N(k^2 + t^2))^\gamma\prod_\rho \frac{|s_1 - \rho|}{|s_0- \rho|}.
\end{align}
Rearranging, we get
\begin{align*}
\frac{|s_1 - \rho|}{|s_0- \rho|} &= \left( 1 - \frac{|s_0 - \rho|^2 - |s_1 - \rho|^2}{|s_0 - \rho|^2}\right)^{1/2} = \left( 1 - \frac{4\gamma\Re(1-\rho)}{|s_0 - \rho|^2}\right)^{1/2}
\intertext{which is a truncation of the Taylor expansion for the exponential function, so}
\frac{|s_1 - \rho|}{|s_0- \rho|}&\leq \exp\left(-\frac{2\gamma\Re(1-\rho)}{|s_0 - \rho|^2} \right)= \exp \left(-2\gamma\Re\left(\frac{1}{s_0 - \rho}\right) + \frac{2\gamma^2}{|s_0 - \rho|^2} \right).
\end{align*}
Substituting this into \eqref{asympL}, we have
\begin{align}\label{asympL2}
\left|\frac{L(s_1,f)}{L(s_0,f)} \right| \asymp (N(k^2 + t^2))^\gamma\prod_\rho \exp \left(-2\gamma\Re\left(\frac{1}{s_0 - \rho}\right) + \frac{2\gamma^2}{|s_0 - \rho|^2} \right).
\end{align}

On the other hand, by taking the logarithmic derivative of \eqref{Hadamardprod} and applying Stirling's formula,  
we obtain 
\begin{align}\label{RelogL}
&-\Re\left(\frac{L'(s_0,f)}{L(s_0,f)}\right) = \frac{1}{2}\log(N(k^2+t^2))  - \sum_\rho\Re\left(\frac{1}{s_0-\rho}\right)+ O(1).
\end{align}
To bound the left-hand side of \eqref{RelogL}, observe that \[-\frac{L'}{L}(s,f) = \sum_{n\geq 1} \frac{\Lambda_f(n)}{n^s},\]  where $\Lambda_f$ is supported on prime powers and satisfies  the identities 
\[\Lambda_f(p)= \lambda_f(p)\log p\hspace{.3in} \text{ and } \hspace{.3in}
\Lambda_f(p^m)= \sum_{j=1}^2 \alpha_{j, f}(p)^m\log p,
\]
where $\alpha_{j, f}(p)$ is the $p$-th local parameter of $f$ as in \eqref{eqn:L-fn}. 
Therefore, 
\begin{align*}
\left| \frac{L'(s_0, f)}{L(s_0,f)}  \right| 
\leq  \sum_{n\geq 1} \left|\frac{\Lambda_f(n)}{n^{1+\gamma + it}}\right| \leq \sum_{p^m} \frac{|\alpha_{1, f}(p)^m+\alpha_{2, f}(p)^m|\log p}{p^{m(1+\gamma)} }\leq 2\sum_{p^m} \frac{\log p}{p^{m(1+\gamma)}} =2\sum_{n\geq 1}\frac{\Lambda(n)}{n^{1+\gamma}},
\end{align*}
where $\Lambda(n)$ is the Von Mangoldt function. 
Therefore, we may write
\begin{align}\label{eqn:logderivative_estimate}
-\Re\left(\frac{L'(s_0,f)}{L(s_0,f)}\right) \leq \frac{2}{\gamma} + O(1).
\end{align}
Applying this bound to \eqref{RelogL} and taking the exponential of both sides give
\begin{align*}
(N(k^2+t^2))^{\gamma} \exp\left( - 2\gamma\sum_\rho\Re\left(\frac{1}{s_0-\rho}\right)\right) \ll 1.
\end{align*}
Going back to \eqref{asympL2}, we have
\begin{align*}
\left|\frac{L(s_1,f)}{L(s_0,f)} \right| &\asymp (N(k^2 + |t|^2))^\gamma\prod_\rho \exp \left(-2\gamma\Re\left(\frac{1}{s_0 - \rho}\right)\right) \prod_\rho\exp\left(  \frac{2\gamma^2}{|s_0 - \rho|^2} \right)\\
&\ll \prod_\rho\exp\left(  \frac{2\gamma^2}{|s_0 - \rho|^2} \right).
\end{align*}
Notice that \[|L(s_0,f)|\leq \sum_{n\geq1}\frac{|\lambda_f(n)|}{n^{1+\gamma}}\leq  \sum_{n\geq1}\frac{\tau(n)}{n^{1+\gamma}}=\zeta^2(1+\gamma)=\left(\frac{1}{\gamma}+O(1)\right)^2,\] and therefore we have the desired result 
\begin{align*}
\left|L(s_1,f) \right|&\ll \frac{1}{\gamma^2}\prod_\rho\exp\left(  \frac{2\gamma^2}{|s_0 - \rho|^2} \right).\qedhere
\end{align*}
\end{proof}

\section{Key Ingredients from Pretentious Number Theory}\label{sec:pretentious}
This section highlights crucial results regarding the mean values of divisor-bounded multiplicative functions. While the works of Granville--Harper--Soundararajan~\cite{GHS1} and Mangerel~\cite{Mangerel} encompass many of these statements, we require specific variations to suit our setting.

For two multiplicative functions $h,g:\N \to \C$, and $x \in \R_{\geq 0}$, we define the distance $\D(h,g;x)^2$ by \[ \D(h,g;x)^2 = \sum_{p \leq x} \frac{1-Re(h(p)\bar{g}(p))}{p}. \] 
In practice, we will only use this notion for $h$ such that $|h(n)|\leq \tau(n)$ and $g(n) = n^{it}$. The distance function is related to the Dirichlet series $H(s)=\sum_{n\geq 1} h(n)n^{-s}$ as follows. Let $x \geq 3$, and let $h$ be such that $|h(p)| \leq 2$ for all $p$ and $h(n) \ll_\epsilon n^{\epsilon}$ uniformly in $n$. Then \[ H\left(1+\frac{1}{\log x} + it\right) \asymp \log x \exp \left(-\D(h,n^{it};x)^2\right) . \] To see this, we use \cite[Lemma 2.2.15]{MangerelThesis} which states that
\begin{equation} \label{Mangerelasymptotics}
\log H\left(1+\frac{1}{\log x} + it\right) = \sum_{p \leq x} \frac{h(p)p^{-it}}{p} + O(1).
\end{equation} 
Here, the error term is independent of $t$. The logarithm in \eqref{Mangerelasymptotics} is taken with respect to the principal branch, so taking real parts gives 
\begin{equation*}
\log \left| H\left(1+\frac{1}{\log x} + it\right) \right|  = \sum_{p \leq x} \frac{\Re(h(p)p^{-it})}{p} + O(1) = \log \log x - \D(h,n^{it},x)^2 + O(1),
\end{equation*} 
where the last asymptotic follows from Mertens' estimate.

Using this distance function, Granville-Harper-Soundararajan prove the following version of  Hal\'{a}sz's Theorem \cite[Corollary 1.2]{GHS1}. 

\begin{thm}[Hal\'asz's Theorem] \label{Halasz} Let $h$ be a multiplicative function such than $|h(n)| \leq \tau(n)$ for all $n\in \N$, and set $H(s)=\sum_{n\geq 1} h(n)n^{-s}$. Then, for $M = \max_{|t|\leq (\log x)^2} |H(1+\frac{1}{\log x} + it)|$, we have
\[
\frac{1}{x}\sum_{n\leq x} h(n) \ll (M+1)e^{-M} \log x + \frac{(\log\log x)^2}{\log x}.
\]
\end{thm}

Throughout the remainder of this section, we will focus exclusively on the case  of $h(n) = \lambda_f(n)$, where  $f \in H_k(N)$. We will commence by stating a version of the so-called Lipschitz Theorem, based on the work of \cite{GHS1}.
\begin{proposition}[Lipschitz Theorem]\label{prop:lipschitz} 
Suppose $f\in H_k(N)$. Let $\phi$ be a number in the range $|t| \leq (\log  x)^2$ for which the function $t \mapsto |L(1+ \frac{1}{\log x} + it, f)|$ reaches its maximum. Then for all $1 \leq w \leq x^{1/3}$
we have
\begin{equation*}
\left| \frac{1}{x} \sum_{n \leq x} \lambda_{f}(n)n^{-i\phi} - \frac{1}{x/w} \sum_{n \leq x/w} \lambda_{f}(n)n^{-i\phi} \right| \ll \log \left(\frac{\log x}{\log ew}\right) \left(\frac{\log w+(\log\log x)^2}{\log x}\right)^{2-\frac{4}{\pi}}\log x. 
\end{equation*}
\end{proposition}
\begin{proof} The proof of this version differs only in very minor details from the proof of \cite[Theorem~ 1.5]{GHS1} where the authors prove that the same upper bound holds for $
\left| \frac{1}{x^{1+i\phi}} \sum_{n \leq x} \lambda_{f}(n) - \frac{1}{(x/w)^{1+i\phi}} \sum_{n \leq x/w} \lambda_{f}(n) \right|
$. The reader is referred to \cite{GHS1} for the detailed exposition.  \end{proof}
We will apply the Lipschitz bound in Proposition \ref{prop:lipschitz} as follows.
\begin{corollary}\label{cor:lipschitz}
Let $f\in H_k(N)$. Let $\phi$ be a number in the range $|t| \leq (\log  x)^2$ for which the function $t \mapsto |L(1+ \frac{1}{\log x} + it, f)|$ reaches its maximum. Then for all $x^{2/3} \leq z \leq x^{3/2}$
	we have
	\begin{equation*}
		\left| \frac{1}{x} \sum_{n \leq x} \lambda_{f}(n)n^{-i\phi} - \frac{1}{z} \sum_{n \leq z} \lambda_{f}(n)n^{-i\phi} \right| \ll \left(\frac{1+ \left|\log\frac{x}{z}\right|}{\log x}\right)^{2-\frac{4}{\pi}+o(1)}\log x. 
	\end{equation*}

\end{corollary}
\begin{proof}
Let $z = x/w$ in Proposition \ref{prop:lipschitz}. Then if $x^{2/3}\leq z \leq x$, we have 
	\begin{align*}
	\left| \frac{1}{x} \sum_{n \leq x} \lambda_{f}(n)n^{-i\phi} - \frac{1}{z} \sum_{n \leq z} \lambda_{f}(n)n^{-i\phi} \right| &\ll \log \left(\frac{\log  x}{\log\frac{ex}{z}}\right)\left(\frac{\log\frac{ex}{z}+(\log \log  x)^2}{\log x }\right)^{2-\frac{4}{\pi}}\mathcal \log x \\
	&\ll \log \log x \left(\frac{\log\frac{ex}{z}+(\log \log x)^2}{\log x}\right)^{2-\frac{4}{\pi}}\log x\\
	& \ll \left(\frac{1+\left|\log\frac{x}{z}\right|}{\log x } \right)^{2-\frac{4}{\pi} + o(1)} \log x.
	\end{align*}
	For the interval $x \leq z \leq x^{3/2}$, we repeat the argument, interchanging the roles of $x$ and $z$.
\end{proof}

We will additionally state the following analogue of \cite[Corollary 3.9]{Mangerel} for later use.
\begin{lemma}\label{cor3.9-mangerel} 
Let $f\in H_k(N)$ and $\phi$ be as in Proposition \ref{prop:lipschitz}.  Then
\begin{equation*}\label{MangerelLemma}
\frac{1}{x}\sum_{n\leq x} \lambda_f(n) = \frac{x^{i\phi}}{1+i\phi}\cdot \frac{1}{x}\sum_{n\leq x} \lambda_f(n) n^{-i\phi} + O\left( (\log x)^{-1+\frac{4}{\pi}}(\log\log x)^{5-\frac{8}{\pi}}\right).
\end{equation*}
\end{lemma}
\begin{proof}
We will show the equivalent statement that  
\[ \frac1x\sum_{n \leq x}\lambda_f(n)n^{-i\phi} = \frac{1+i\phi}{x^{1+i\phi}}\sum_{n \leq x}\lambda_f(n) + O\left( |\phi| (\log x)^{-1+\frac{4}{\pi}}(\log\log x)^{5-\frac{8}{\pi}}\right). \]
By partial summation, we have
\begin{align*}
\frac1x\sum_{n \leq x}\lambda_f(n)n^{-i\phi} &= \frac1x\int_1^x u^{-i\phi} d\left\{ \sum_{n\leq x} \lambda_f(n)\right\}= \frac{1}{x^{1+i\phi}}\sum_{n\leq x} \lambda_f(n) + \frac{i\phi}{x}\int_1^x \frac{1}{u^{1+i\phi}}\sum_{n\leq u}\lambda_f(n)du.
\end{align*}
We split the integral into two pieces as follows:
\begin{align*}
\int_1^x \frac{1}{u^{1+i\phi}}\sum_{n\leq u}\lambda_f(n)du &= \int_1^{x/(\log x)^2} \frac{1}{u^{1+i\phi}}\sum_{n\leq u}\lambda_f(n)du + \int_{x/(\log x)^2}^x \frac{1}{u^{1+i\phi}}\sum_{n\leq u}\lambda_f(n)du.
\end{align*}
For the first integral, we use the trivial bound, so
\begin{align*}
\frac{i\phi}{x}\int_1^{x/(\log x)^2} \frac{1}{u^{1+i\phi}}\sum_{n\leq u}\lambda_f(n)du &\ll \frac{|\phi|}{x}\int_1^{x/(\log x)^2} \frac{1}{u}\sum_{n\leq u}|\lambda_f(n)|du \leq \frac{|\phi|}{x}\int_1^{x/(\log x)^2} \log u \, du 
 \leq \frac{|\phi|}{\log x}.
\end{align*}
Since $w = x/u$ is in the range of Proposition ~\ref{prop:lipschitz}, the second integral is equal to 
\begin{align*}
&\frac{i\phi}{x}  \int_{x/(\log x)^2}^x  \left(  \frac{1}{x^{1+i\phi}}\sum_{n\leq x}\lambda_f(n) + O\left( \log \left(\frac{\log x}{\log ew}\right) \left(\frac{\log w+(\log\log x)^2}{\log x}\right)^{2-\frac{4}{\pi}}\log x  \right)\right)du\\
 = &\frac{i\phi}{x}\left( \frac{1}{x^{1+i\phi}}\sum_{n\leq x}\lambda_f(n)  \right)\int_{x/(\log x)^2}^x du +  \frac{i\phi}{x}\cdot O\left( \log\log x \left(\frac{(\log\log x)^2}{\log x}\right)^{2-\frac{4}{\pi}}\log x\right)\int_{x/(\log x)^2}^x du\\
= & \frac{i\phi}{x^{1+i\phi}}\sum_{n\leq x}\lambda_f(n)  + O\left(|\phi| \frac{(\log\log x)^{5-\frac{8}{\pi}}}{(\log x)^{2-\frac{4}{\pi}}}\log x \right).
\end{align*}
Combining the two integrals gives the desired result. 
\end{proof}
We conclude with a result that will play a pivotal role in the next section. 
\begin{proposition}\label{mvfbound}
Let $f\in H_k(N)$. Let $\phi \in [-(\log x)^2, (\log x)^2]$ be the point at which the maximal value  $M = \max_{|t|\leq (\log x)^2} |L(1+\frac{1}{\log x} + it, f)|$ is attained. Then,
\[
\frac{1}{x}\sum_{\n\leq x} \lambda_f(n) \ll \log x \left(\frac{(M+1)e^{-M}}{1+|\phi|} + \frac{(\log\log x)^{5-\frac{8}{\pi}}}{(\log x)^{2-\frac{4}{\pi}}}\right).
\]
\end{proposition}
\begin{proof}

Using Theorem~\ref{Halasz}, we have
\[
\frac{1}{x}\sum_{n\leq x} \lambda_f(n)n^{-i\phi} \ll (M+1)e^{-M} \log x + \frac{(\log\log x)^2}{\log x}.
\]
Applying Lemma~\ref{cor3.9-mangerel}, we get
\begin{align*}
\frac{1}{x}\sum_{n\leq x} \lambda_f(n) &\ll  \left|\frac{x^{i\phi}}{1 + i\phi}\right|\left((M+1)e^{-M} \log x + \frac{(\log\log x)^2}{\log x}\right) +  O\left( \frac{(\log\log x)^{5-\frac{8}{\pi}}}{(\log x)^{1- \frac{4}{\pi}}}\right)\\
&\ll  \frac{(M+1)e^{-M}}{1 + |\phi|}  \log x + O\left(\frac{(\log\log x)^2}{\log x}\right) +  O\left(\frac{(\log\log x)^{5-\frac{8}{\pi}}}{(\log x)^{1- \frac{4}{\pi}}}\right).
\end{align*}
The first error term is smaller than the second, 
so it is subsumed into the second one. 
\end{proof}

\section{Necessary Lemmas}\label{sec:lemmas}

In this section, we establish the key ingredients required for the proof of the main theorem; namely Lemmas~\ref{lem:estimate average}~and~\ref{lem:L-average}, and Proposition~\ref{prop}. Across all the statements, we assume $f\in H_k(N)$.

\begin{lemma} \label{lem:estimate average}
Let $y_0 \geq 4$ and assume that $|S(e^{y_0}, f)| \geq y_0e^{y_0}y_0^{-1/100}$. 
There exists a real number $\phi = \phi(y_0)$ with $|\phi| \ll y_0e^{y_0}/|S(e^{y_0},f)|$ such that, for any $y\in \R$, 
\begin{equation*}\label{eqn:cor-lipschitz1} 
\left| \frac{S(e^y, f, \phi)}{e^y} - \frac{S(e^{y_0}, f, \phi)}{e^{y_0}}  \right| \ll \left( \frac{\log y_0+ |y-y_0|}{y_0} \right)^{2-\frac{4}{\pi}+o(1)}\max\{y, y_0\}.
\end{equation*}
Moreover, for any $\epsilon >0$, we have
 \begin{equation}\label{eqn:cor-lipschitz2}
	S(e^{y_0}, f, \phi) = (1+i\phi)e^{-i\phi y_0}S(e^{y_0}, f)+ O(e^{y_0}y_0^{-1+\frac{4}{\pi}+\epsilon}).
\end{equation}
\end{lemma}
\begin{proof}
	We set $x=e^{y_0}$ in Proposition~\ref{mvfbound}. There exists $\phi \in [-(y_0)^2, (y_0)^2]$ such that
\begin{align}\label{eqn:mvfbound}
\frac{1}{e^{y_0}}\sum_{n\leq e^{y_0}} \lambda_f(n) \ll \frac{(1+M)e^{-M} y_0}{1+|\phi|} + y_0^{-1+\frac{4}{\pi}}(\log y_0)^{5-\frac{8}{\pi}}.
\end{align}
We first show that $1+|\phi|\ll y_0e^{y_0}/|S(e^{y_0},f)|$. Starting from \eqref{eqn:mvfbound}, we get
\[  \frac{|S(e^{y_0}, f)|}{y_0e^{y_0}} \leq C\left( \frac{(1+M)e^{-M}}{1+|\phi|} + \frac{1}{y_0^{2-\frac{4}{\pi}-\epsilon}} \right) , \] for some absolute positive constant $C$.
Rearranging the above equation yields
\[ 1+|\phi| \leq \frac{C(M+1)e^{-M}}{\frac{|S(e^{y_0}, f)| }{y_0e^{y_0}} -C y_0^{-2+\frac{4}{\pi}+\epsilon}} \leq \frac{2C(M+1)e^{-M}}{\frac{|S(e^{y_0}, f)| }{y_0e^{y_0}}} \] 
where the last inequality holds as long as 
\[ \frac{|S(e^{y_0}, f)| }{y_0e^{y_0}} \geq 2C y_0^{-2+\frac{4}{\pi}+\epsilon},  \] 
which happens for all $y_0 \gg_C 1$ since $\frac{|S(e^{y_0}, f)| }{y_0e^{y_0}} \geq  y_0^{-1/100}$. 
Hence,
\[1+|\phi|\ll  \frac{y_0e^{y_0}}{|S(e^{y_0}, f)|}.\]
The first assertion follows from Corollary~\ref{cor:lipschitz} with $x=e^{y_0}$ and $z=e^{y}$. The second assertion follows from Lemma~\ref{cor3.9-mangerel} by taking  $x=e^{y_0}$. 
\end{proof}

\begin{lemma}\label{lem:L-average}
Let $\phi$ be a real number, $T$ a positive real number, and $\gamma$ a real number such that $0\leq \gamma\leq\frac{1}{2}$. 
Set $S(x, f, \phi) \vcentcolon=\sum_{n\leq x}\lambda_f(n)n^{-i\phi}$. Then
\begin{align*}
\sqrt{2\pi T}\int_{-\infty}^\infty \frac{S(e^y,f,\phi)}{e^y}\exp\left(\gamma y - \frac{T}{2}y^2\right)dy = \int_{-\infty}^\infty \frac{L(1-\gamma + i\phi + i\xi,f)}{1-\gamma+i\xi} \exp \left(-\frac{\xi^2}{2T}\right)d\xi.
\end{align*}
\end{lemma}

\begin{proof}
We apply Plancherel's formula with $g(y)=\frac{S(e^y,f,\phi)}{e^y}\exp(\gamma y )$ and  $h(y)=\exp\left(-\frac{T}{2}y^2\right)$.
The result follows by computing
\begin{align*}
\widehat{g}(\xi) &= \int_{-\infty}^\infty \frac{S(e^y,f,\phi)}{e^y}\exp(\gamma y -i\xi y)dy\\
&= \sum_{n\leq 1} \lambda_f(n) n^{-i\phi} \int_{\log n}^\infty e^{y(\gamma - 1-i\xi)} dy\\
&= \sum_{n\leq 1} \lambda_f(n) n^{-i\phi}  \frac{n^{\gamma-1-i\xi}}{1-\gamma+i\xi} 
=\frac{L(1-\gamma +i\phi + i\xi,f)}{1 - \gamma + i\xi},
\end{align*}
and
\begin{align*} 
\widehat{h}(\xi)=\int_{-\infty}^\infty \exp\left(-\frac{T}{2}y^2 - iy\xi\right)dy = \sqrt{\frac{2\pi}{T}}\exp \left( - \frac{\xi^2}{2T}\right).
\end{align*}
\end{proof}

The next proposition combines the previous two lemmas to derive a lower bound for a certain sum over non-trivial zeros of $L(s,f)$ under the assumption that $S(x,f)$ is large. 

\begin{proposition}\label{prop}
Let $y_0$ be large with $|S(e^{y_0}, f)|=\vcentcolon y_0e^{y_0}/Q\geq y_0e^{y_0}y_0^{-1/100}$, and let $\phi$ be as in Lemma~\ref{lem:estimate average}. If $CQ^3/y_0\leq\gamma\leq 2/5$ for a suitably large constant $C$, then there exists $|\xi|\leq 2\sqrt{\gamma\log (k^\gamma\log k)/y_0}$ such that
\[ \sum_\rho \frac{\gamma}{|1+\gamma+i(\phi+\xi)-\rho|^2}\geq \frac{y_0}{4}.\]
\end{proposition}
\begin{proof}
Set $T=\gamma/y_0$, and note that $T\leq 1$. 
Applying Lemma~\ref{lem:estimate average}, we have that
\begin{align*}
&\sqrt{2\pi T}\int_{-\infty}^\infty \frac{S(e^y, f, \phi)}{e^y}\exp\left(\gamma y - \frac{T}{2}y^2\right)dy \\
&\hspace{.2in}= \sqrt{2\pi T}\exp\left(\frac{\gamma y_0}{2}\right)\int_{-\infty}^\infty \left(\frac{S(e^{y_0}, f, \phi)}{e^{y_0}}+O\left(\frac{\log y_0+|y-y_0|^{2/3}}{y_0^{2/3}} y_0 \right)  \right)\exp\left(-\frac{T}{2}(y-y_0)^2\right)dy \\
& \hspace{.2in}= \sqrt{2\pi T}\exp\left(\frac{\gamma y_0}{2}\right)\left(\frac{(1+i\phi)S(e^{y_0}, f)}{e^{y_0(1+i\phi)}}+O\left(y_0^{1/3}\right) \right) \int_{-\infty}^\infty \exp\left(-\frac{T}{2}(y-y_0)^2\right)dy \\
& \hspace{.9in} +O\left(y_0^{1/3}\sqrt{2\pi T}\exp\left(\frac{\gamma y_0}{2}\right)\int_{-\infty}^{\infty}\left((\log y_0+|y-y_0|^{2/3}) \right)\exp\left(-\frac{T}{2}(y-y_0)^2\right)dy\right).
\end{align*}
Noting that
\[  \int_{-\infty}^\infty \exp\left(-\frac{T}{2}(y-y_0)^2\right)dy=\sqrt{\frac{2\pi}{T}}\]
and 
\[\int_{-\infty}^{\infty}\left((\log y_0+|y-y_0|^{2/3}) \right)\exp\left(-\frac{T}{2}(y-y_0)^2\right)dy=\log y_0\sqrt{\frac{2\pi}{T}}+\left(\frac{2}{T}\right)^{5/6}\Gamma\left(\frac56\right),\]
the integral above equals
\begin{align*}
&2\pi \exp\left(\frac{\gamma y_0}{2}\right)\left(\frac{(1+i\phi)S(e^{y_0}, f)}{e^{y_0(1+i\phi)}}+O\left(y_0^{1/3}+ y_0^{1/3}\log y_0+\left(\frac{y_0}{T}\right)^{1/3}\right) \right) \\
=&2\pi \exp\left(\frac{\gamma y_0}{2}\right)\left(\frac{(1+i\phi)S(e^{y_0}, f)}{e^{y_0(1+i\phi)}}+O\left(\frac{y_0^2}{\gamma}\right)^{\frac13}\right). 
\end{align*}
Using our lower bound on $\gamma$, we see that this integral is in magnitude $\geq \pi y_0 \exp\left(\frac{\gamma y_0}{2}\right)/Q$. So it follows from our assumption on~$S(e^{y_0},f)$ and Lemma~\ref{lem:L-average} that
\begin{align*}
\frac{ \pi y_0}{Q} \exp\left(\frac{\gamma y_0}{2}\right) &\leq \int_\R\frac{|L(1-\gamma+i(\phi+\xi), f)|}{|1-\gamma+i\xi|}\exp\left(-\frac{\xi^2}{2T}\right)\, d\xi \\
 &\leq \left(\max_{\xi\in \R}\frac{|L(1-\gamma+i(\phi+\xi), f)|}{|1-\gamma+i\xi|}\exp\left(-\frac{\xi^2}{4T}\right)\right)\int_\R\exp\left(-\frac{\xi^2}{4T}\right)\, d\xi.
 \end{align*}
 Thus,
 \begin{equation}\label{eqn:L-ratio}
 \max_{\xi\in \R}\frac{|L(1-\gamma+i(\phi+\xi), f)|}{|1-\gamma+i\xi|}\exp\left(-\frac{\xi^2}{4T}\right) \geq \frac{\pi y_0 \exp\left(\frac{\gamma y_0}{2}\right)}{Q}\frac{1}{2\sqrt{T\pi}}
 =\frac{y_0\exp\left(\frac{\gamma y_0}{2}\right)}{2Q}\sqrt{\frac{\pi y_0}{\gamma}}.
 \end{equation}
 If $\Re(s)=\sigma>1/2$, we have that
 \begin{align*}
 |L(s, f)|&\leq \left|s\int_1^\infty \frac{S(x, f)}{x^{s+1}}\, dx\right|\leq |s|\int_1^\infty \frac{\min(x\log x, (xk)^{1/2}\log (xk)^{1/2})}{x^{\sigma+1}}\, dx \\
 &=\frac{|s|}{k^{\sigma-1}}\left(\frac{k^\sigma}{(\sigma-1)^2}-\frac{\log k}{\sigma-1}-\frac{1}{(\sigma-1)^2}+\frac{2\log k}{2\sigma-1}+\frac{2}{(2\sigma-1)^2}\right),
 \end{align*}
where we recall that $k$ is the weight of $f$. It follows that, if $|\xi|>2\sqrt{\gamma\log (k^\gamma\log k)/y_0}$, we have 
\begin{align*}
&\frac{|L(1-\gamma+i(\phi+\xi), f)|}{|1-\gamma+i\xi |}\exp\left(-\frac{\xi^2}{4T}\right)\\
&\hspace{.7in} \leq \left|\frac{1-\gamma+i(\phi+\xi)}{1-\gamma+i\xi}\right| k^\gamma\left(\frac{k^{-\gamma}}{\gamma^2}+\frac{\log k}{\gamma}-\frac{1}{\gamma^2}+\frac{2\log k}{1-2\gamma}+\frac{2}{(1-2\gamma)^2}\right)\exp\left(-\frac{\xi^2}{4T}\right) \\
&\hspace{.7in}= \left|\frac{1-\gamma+i(\phi+\xi)}{1-\gamma+i\xi}\right|\left(\frac{k^\gamma\log k}{\gamma}+\frac{2k^\gamma\log k}{1-2\gamma}-\frac{k^\gamma-1}{\gamma^2}+\frac{2k^\gamma}{(1-2\gamma)^2}\right)\exp\left(-\frac{\xi^2}{4T}\right) \\
&\hspace{.7in}\leq  \left|\frac{1-\gamma+i(\phi+\xi)}{1-\gamma+i\xi}\right|\left(\frac{k^\gamma\log k}{\gamma}+\frac{3k^\gamma\log k}{1-2\gamma}\right)\frac{k^{-\gamma}}{\log k}\\
&\hspace{.7in}\leq \frac{7(1+2|\phi|)}{\gamma}.  
\end{align*}
Since $CQ^3/y_0\leq \gamma$, the right-hand side of (\ref{eqn:L-ratio}) is greater than 
\[\frac{\pi^{1/2}C^{\frac32}Q^{\frac{7}{2}}\exp(\frac{CQ^3}{2})}{\gamma^2}\]
which is larger than $(7+14Q)/\gamma$ with a suitably large $C$. Therefore, the maximum on the left-hand side of \eqref{eqn:L-ratio} cannot be attained  in this range of $\xi$. Thus, there exists $\xi$ with $|\xi|\leq 2\sqrt{\gamma\log (k^\gamma\log k)/y_0}$ such that
\begin{equation*}
 |L(1-\gamma+i(\phi+\xi), f)|\geq \frac{y_0}{2Q}\exp\left(\frac{\gamma y_0}{2}\right)\sqrt{\frac{\pi y_0}{\gamma}} |1-\gamma+i\xi|\exp\left(\frac{\xi^2}{4T}\right)
 \geq \frac{ y_0}{2Q}\exp\left(\frac{\gamma y_0}{2}\right)\sqrt{\frac{\pi y_0}{\gamma}}.
 \end{equation*}
By utilizing this bound in conjunction with Lemma~\ref{lem:L-upperbound}, we obtain 
\begin{align*}
 \frac{1}{\gamma^2} \exp\left(\sum_\rho \frac{2\gamma^2}{|1+\gamma+ i(\phi+\xi) - \rho|^2} \right)\gg  \frac{ y_0}{2Q}\exp\left(\frac{\gamma y_0}{2}\right)\sqrt{\frac{\pi y_0}{\gamma}}.
\end{align*}
Consequently, 
\begin{align*}
\exp\left(2\gamma\left(\sum_\rho \frac{\gamma}{|1+\gamma + i(\phi+\xi) - \rho|^2} - \frac{y_0}{4}\right)\right)&\gg \frac{\gamma^2y_0}{2Q}\sqrt{\frac{\pi y_0}{\gamma}}\geq \frac12\sqrt{\pi}C^{\frac32}Q^{\frac72},
\end{align*}
since $\gamma y_0\geq CQ^3$. Choosing $C$ large enough ensures the right-hand side is $\geq 1$. 
Hence,
 \[ \sum_\rho\frac{\gamma}{|1+\gamma+i(\phi+\xi)-\rho|^2}-\frac{y_0}{4}\geq0,\]
 as desired. 
\end{proof}

\section{Proof of Theorem \ref{thm:main}}\label{sec:proof}

This section is devoted to the proof of Theorem~\ref{thm:main}. Let $\phi, \gamma$, and $\xi$ be as given in Proposition~\ref{prop}. Let $\dis Y=\sqrt{\frac{\log(k^\gamma\log k)}{y_0}}$, and suppose that $|1+i\phi-\rho|\geq 100Y^2$. Then,
\begin{align*}
|1+\gamma+i(\phi+\xi)-\rho|&\geq \left|1+50Y^2+i\phi-\rho\right|-\left(50Y^2+|\xi|\right) \\
&\geq \left|1+50Y^2+i\phi-\rho\right|-\left(50Y^2+2Y\right) \\
&\geq \frac12\left|1+50Y^2+i\phi-\rho\right|.
\end{align*}
Therefore, applying \eqref{RelogL} and \eqref{eqn:logderivative_estimate}, along with the bound $|\phi|\ll Q\leq y_0^{1/100}$, we have
\begin{align}
\sum_{\substack{\rho \\ |1+i\phi-\rho|>100Y^2}}\frac{\gamma}{ |1+\gamma+i(\phi+\xi)-\rho|^2} &\leq \frac{4\gamma}{50Y^2}\sum_{\substack{\rho \\ |1+i\phi-\rho|>100Y^2}}\frac{1}{\left|1+50Y^2+i\phi-\rho\right|}\nonumber\\
&\leq \frac{2\gamma}{25Y^2}\sum_{\substack{\rho \\ |1+i\phi-\rho|>100Y^2}}\Re\left(\frac{1}{1+50Y^2+i\phi-\rho}\right)\nonumber\\
&\leq \frac{2\gamma}{25Y^2}\left(\frac12\log (k^2+y_0^{1/50})+\frac{1}{100Y^2}+O(1)\right)  \nonumber\\
&\leq \frac{2 y_0}{25\log k}\left(\frac12\log (k^2+y_0^{1/50})+\frac{y_0}{100\gamma\log k}+O(1)\right).\nonumber
\end{align}
Letting $y_0=\log x$, with $\sqrt{\log k}\leq \log x \leq \frac12\log k$,  we get
\begin{align}\label{eqn:upperbd}
\sum_{\substack{\rho \\ |1+i\phi-\rho|>100Y^2}}\frac{\gamma}{ |1+\gamma+i(\phi+\xi)-\rho|^2} &\leq \frac{2\log x}{25\log k}\left(\frac12\log (2k^2)+\frac{(\log x)^2}{100CQ^3\log k}+O(1)\right)\nonumber\\
&\leq \frac{2 \log x}{25\log k}\left(\log k+\frac{\log x}{200CQ^3}+O(1)\right)\nonumber\\
& \leq \frac{2}{25}\left(\log x+\frac{\log x}{400CQ^3}+O(1)\right)\nonumber\\
& \leq \frac{2}{25}\left(\log x+A\log x \right)\leq \frac{9}{100}\log x,
\end{align}
where the last inequality follows from taking $A$ small enough, which is possible by choosing $C$ sufficiently large. Using Proposition \ref{prop} and \eqref{eqn:upperbd} gives
\begin{align}\label{eqn:zero_sum}
\sum_{\substack{\rho \\ |1+i\phi-\rho|\leq 100Y^2}}\frac{\gamma}{ |1+\gamma+i(\phi+\xi)-\rho|^2}&=\sum_{\rho}\frac{\gamma}{ |1+\gamma+i(\phi+\xi)-\rho|^2}-\sum_{\substack{\rho \\ |1+i\phi-\rho|> 100Y^2}}\frac{\gamma}{ |1+\gamma+i(\phi+\xi)-\rho|^2}\\
&\geq \frac{\log x}{4}-\frac{9}{100}\log x=\frac{4}{25}\log x. \nonumber
\end{align}
Since 
\[ \frac{\gamma}{|1+\gamma+i(\phi+\xi)-\rho|^2}\leq \frac{1}{\gamma},\]
the left-hand side of \eqref{eqn:zero_sum} is at most $1/\gamma \cdot \#\{\rho\,:\, |1+i\phi-\rho|<100Y^2\}$. 
Recalling that 
\[ Y^2=\frac{\log(k^\gamma\log k)}{\log x}=\frac{\gamma\log(k(\log k)^{1/\gamma})}{\log x},\]
we conclude that
\[ \#\left\{\rho\,:\, |1+i\phi-\rho|<\frac{100\gamma\log(k(\log k)^{1/\gamma})}{\log x}\right\}\geq \frac{4\gamma\log x}{25}.\]
The proof of the theorem is completed by setting $L\vcentcolon=100\gamma\log x$.

\section*{Acknowledgements}
This project originated from Women~In~Numbers~6 Research Workshop that took place at Banff International Research Station in March 2023. The authors express their utmost gratitude to the organizers for the invaluable opportunity provided by the workshop.

\bibliographystyle{siam}
\bibliography{references}

\begin{thebibliography}{10}

\bibitem{Deligne}
{\sc P.~Deligne}, {\em La conjecture de {W}eil. {I}}, Inst. Hautes \'{E}tudes
  Sci. Publ. Math.,  (1974), pp.~273--307.

\bibitem{GHS1}
{\sc A.~Granville, A.~J. Harper, and K.~Soundararajan}, {\em A new proof of
  {H}al\'{a}sz's theorem, and its consequences}, Compos. Math., 155 (2019),
  pp.~126--163.

\bibitem{GS1}
{\sc A.~Granville and K.~Soundararajan}, {\em Large character sums}, J. Amer.
  Math. Soc., 14 (2001), pp.~365--397.

\bibitem{GS3}
\leavevmode\vrule height 2pt depth -1.6pt width 23pt, {\em Large character
  sums: {B}urgess's theorem and zeros of {$L$}-functions}, J. Eur. Math. Soc.
  (JEMS), 20 (2018), pp.~1--14.

\bibitem{Hafner-Ivic}
{\sc J.~L. Hafner and A.~Ivi\'{c}}, {\em On sums of {F}ourier coefficients of
  cusp forms}, Enseign. Math. (2), 35 (1989), pp.~375--382.

\bibitem{hecke}
{\sc E.~Hecke}, {\em Theorie der {E}isensteinsche reihen h{\"{o}}herer stufe
  und ihre anwendung auf funktionentheorie und arithmetik}, Abh. Math. Semin.
  Univ. Hambg., 5 (1927), pp.~199--224.

\bibitem{IwaniecKowalski}
{\sc H.~Iwaniec and E.~Kowalski}, {\em Analytic number theory}, vol.~53 of
  American Mathematical Society Colloquium Publications, American Mathematical
  Society, Providence, RI, 2004.

\bibitem{Lamzouri}
{\sc Y.~Lamzouri}, {\em Large sums of {H}ecke eigenvalues of holomorphic cusp
  forms}, Forum Math., 31 (2019), pp.~403--417.

\bibitem{MangerelThesis}
{\sc A.~P. Mangerel}, {\em Topics in multiplicative and probabilistic number
  theory}, University of Toronto (Canada), 2018.

\bibitem{Mangerel}
{\sc A.~P. Mangerel}, {\em Divisor-bounded multiplicative functions in short
  intervals}, Research in the Mathematical Sciences, 10 (2023), pp.~1--47.

\bibitem{walfisz}
{\sc A.~Walfisz}, {\em {\"U}ber die {K}oeffizientensummen einiger
  {M}oduformen}, Math. Ann., 108 (1933), pp.~75--90.

\end{thebibliography}

\end{document}